\title{Properties of solutions to some weighted $p$-Laplacian equation}
\author{
        Prashanta Garain
        }
\documentclass{article}
\usepackage{graphicx}
\usepackage{hyperref,amsfonts,amsmath,amsthm,bigints}
\newtheorem{The}{Theorem}[section]
\newtheorem{Lem}{Lemma}[section]
\newtheorem{Rem}{Remark}[section]

\newtheorem{Def}{Definition}[section]

\newenvironment{AMS}{}{}
\newenvironment{keywords}{}{}
\begin{document}
\newpage
\maketitle

\begin{abstract}
In this paper, we prove some qualitative properties for the positive solutions to some degenerate elliptic equation 
given by \[-\operatorname{div}(w|\nabla u|^{p-2}\nabla u)=f(x,u);\;\;w\in \mathcal{A}_p\] on smooth domain 
and for varying nonlinearity $f$.
\end{abstract}


\section*{Introduction}
In this paper we are interested in studying the non-negative solutions to the degenerate elliptic equation given by 
\[-\Delta_{p,w}u=f(x,u);\;\;w\in \mathcal{A}_p\] where the operator $-\Delta_{p,w}$ is defined by \[-\Delta_{p,w}u=-\operatorname{div}(w|\nabla u|^{p-2}\nabla u)\] 
either in whole space or in smooth and bounded domain $\Omega(\subset \mathbb{R}^N)$ for varying degree of nonlinearity.  

We start by proving some preliminary result which guarantees the existence of certain test function which is an important ingredient in proving the main results. 
In our case, because of the weight, desirable regularity results are not readily available. Hence we prove some lemmas to overcome this difficulty.
Then for a wide class of functions $f$ satisfying some conditions we prove a non-existence result, comparison principle, Hardy-type inequality and Liouville theorem.
We also study the eigenvalue problem corresponding to the operator $-\Delta_{p,w}$ proving results such as simplicity and monotonicity. 
Beginning with the paper of Fabes et al \cite{EFabes} where a local H\"older regularity result along with some maximum principle and Poincar\'e type inequalities are proved for Muckenhoupt weights there 
has been a huge surge in interest to study problems on degenerate elliptic operators with Muckenhoupt weights. Related results on  weighted Poincar\'e and Sobolev inequalities was obtained by Chanillo and Wheeden \cite{ChWe}  .
In De Cicco-Vivaldi \cite{DeVi} a Liouville theorem was proved for the weight $w(x)=|x|^r$ with $r>-N$ and $N>2$. Related degenerate eigenvalue problem was studied in Kawohl et al \cite{KLP}.
For more information on this field one can refer to Heinonen et al \cite{JTO}.

The goal of our paper is to improve and complement the previously obtained results where ever possible.
Throughout the paper, we restrict the weight function in the class $\mathcal{A}_p$ (defined below) and choose $\Omega$ to be any bounded smooth domain in $\mathbb{R}^N$ where $N\geq 1$ and $p>1$ unless otherwise stated.   

\section{Preliminaries}
We begin this section by presenting some facts about the weighted Sobolev space. 
\begin{Def} 
A weight function $w$ on any domain $\Omega$ of $\mathbb{R}^N$ is a real-valued function which is measurable and positive a.e.
\end{Def}
\begin{Def} The class of weights $\mathcal{A}_p$ is defined by
\[
\mathcal{A}_p=\big\{a(x)\;\text{is a weight}: a\in L^1_{loc}(\Omega),\;a^{-\frac{1}{p-1}}\in L^1_{loc}(\Omega) 
\big\}
\]
\end{Def}
\begin{itemize}
\item $w(x)=|x|^\alpha\in \mathcal{A}_p$ if $-N<\alpha<N(p-1)$.
\end{itemize}

\begin{Def}(Weighted Sobolev space)
For $w\in\mathcal{A}_p$, we define the weighted Sobolev space $W^{1,p}(\Omega,w)$ to be the set of all real-valued functions $u$ defined a.e. on $\Omega$ for which 
\begin{equation}\label{norm1}
||u||_{1,p,w}=(\int_{\Omega}|u(x)|^{p}dx+\int_{\Omega}w(x)|\nabla u(x)|^{p}dx)^\frac{1}{p}<+{\infty}.
\end{equation}
\end{Def}

\begin{Rem}
The fact $w\in L^{1}_{loc}(\Omega)$ implies $C_{c}^{\infty}(\Omega)\subset W^{1,p}(\Omega,w)$. As a consequence one can introduce the space 
$$
W_{0}^{1,p}(\Omega,w)=\overline{\{C_{c}^{\infty}(\Omega),||.||_{1,p,w}\}}
$$
Moreover, for $w\in \mathcal{A}_p$ both the spaces $W^{1,p}(\Omega,w)$ and $W_{0}^{1,p}(\Omega,w)$ are uniformly convex (and hence reflexive) Banach space, for details see Drabek et al \cite{Ak}.
\end{Rem}
Unfortunately to prove our results suitable embedding theorems were not available for weights in 
$\mathcal{A}_p$ class. Hence we start by defining some more subclass.
\begin{Def}
Define a new sub-class of $\mathcal{A}_p$ as follows: 
$$
\mathcal{A}_s=\big\{a(x)\in \mathcal{A}_{p}:\;\mbox{for some s}\geq\frac{1}{p-1},\;a^{-s}\in L^1(\Omega) \big\}
$$
\end{Def} 
\begin{itemize}
\item For $p>k\geq 1$, $s=\frac{1}{p-k}$, we have $|x|^\alpha\in \mathcal{A}_s$ if $-N<\alpha<N(p-k)$.
\end{itemize}

\begin{Rem}
If we denote by $p_s=\frac{ps}{s+1}$ for $s\geq\frac{1}{p-1}$.
\end{Rem}
Regarding the embedding we have the following Lemma which follows from Drabek et al \cite{Ak}.
\begin{Lem}(Embedding Results)\label{embedding theorem}\\
\begin{itemize}
\item For any $w \in A_s$, we have the continuous inclusion map
\[
    W^{1,p}(\Omega,w)\hookrightarrow W^{1,p_s}(\Omega)\hookrightarrow 
\begin{cases}
    L^q(\Omega),& \text{for } q\in[p_s,p_s^{*}], \text{in case of } 1\leq p_s<N \\
    L^q(\Omega),& \text{for } q\in[p_s,\infty), \text{in case of } p_s=N \\
    C^{0,\alpha}(\overline{\Omega}),& \text{in case of } p_s>N.
\end{cases}
\]
for some $\alpha>0$ and $p_s^{*} = \frac{Np_s}{N-p_s}$.

\item Moreover, these are compact except for $q=p_s^{*}$ in case of $1\leq p_s<N$. 
\item The same result holds for the space $W_{0}^{1,p}(\Omega,w)$.

\end{itemize}
\end{Lem} 

\begin{Def}
We define a new class of weights 
$$
\mathcal{A}_{t}=\{a(x)\in \mathcal{A}_{p}:\text{for some}\;s\in(\frac{N}{p},\infty)\cap[\frac{1}{p-1},\infty),\;a^{-s}\in L^1(\Omega)\}
$$
\end{Def}
\begin{itemize}
\item Assume $p>2N$. Then for $s=\frac{N}{p-2N}$ we have $p_s>N$. Hence $|x|^\alpha\in \mathcal{A}_t$ with $p_s>N$ if $-N<\alpha<p-2N$.
\end{itemize}
\begin{Rem} 
 For $w\in \mathcal{A}_t$, we have the compact embedding $W^{1,p}(\Omega,w)\hookrightarrow L^{p+\eta}(\Omega)\label{12}$ for $0\leq\eta<p_s^{*}-p$. 
 The same holds, if we replace the space $W^{1,p}(\Omega,w)$ by $W_{0}^{1,p}(\Omega,w)$.
\end{Rem}
For more information regarding the weighted Sobolev space see \cite{Ak, JTO} and the references therein.

\textbf{Notation:}\\
(i) We denote the space $W_{0}^{1,p}(\Omega,w)$ by $X$.\\
(ii) For any set $S$ we denote by $S^+$ the set of all nonnegative elements in $S$.\\
(iii) $|S|$ will denote the Lebesgue measure of $S$.\\
(iv) We write $C$ to denote a positive constant which may vary from line to line or even in the same line depending on the situation.

\section{Essential Lemmas}
We start this section by proving some lemmas.\\
For the rest of the paper, we will assume $h\in \mathbb{Q}$ and  $f\in\mathbb{M}$ unless otherwise stated where $\mathbb{Q}$ and $\mathbb{M}$ are defined as follows:
\begin{itemize}
\item $\mathbb{Q}=\big\{ g:(0,\infty)\to(0,\infty):g\;\;\text{is a}\;\;C^1\;\;\text{function}\big\}$
\item $\mathbb{M}=\big\{ f\in \mathbb{Q}: f'(y)\geq (p-1)[f(y)]^{\frac{p-2}{p-1}}]\big\}$
\end{itemize}
Clearly $f(x)=(x+c)^{p-1}$ belongs to $\mathbb{M}$ for any $c\geq 0$ and the equality is achieved. Other example of functions in $\mathbb{M}$ include $g(x)=e^{(p-1)x}$.

\begin{Lem}(Picone Identity)\label{GPI}
Let $\Omega$ be any domain in $\mathbb{R}^N$ and $u\geq{0},\;v>{0}$ in $\Omega$ be differentiable functions and $w$ be a weight function defined in $\Omega$. Define
\begin{gather*}
L(u,v)=w(x)(|\nabla{u}|^{p}-\frac{pu^{p-1}}{f(v)}|\nabla{v}|^{p-2}\nabla{u}.\nabla{v}+\frac{u^{p}f^{'}(v)}{(f(v))^{2}}|\nabla{v}|^{p})\\
R(u,v)=w(x)(|\nabla{u}|^{p}-\nabla(\frac{u^p}{f(v)})|\nabla{v}|^{p-2}\nabla{v})
\end{gather*}
Then, we have $$L(u,v)=R(u,v)\geq{0}\;in\;\;\Omega.$$\\
Moreover, we have $$L(u,v)=0\;\;\iff\;\;u=cv+d\;\mbox{for some constants $c,d.$}$$
\end{Lem}
\begin{proof}
By expanding $R(u,v)$ we get the equality $L(u,v)=R(u,v)$.\\
Note that, 
\begin{multline*}
L(u,v)=w(x)(|\nabla u|^p-\frac{p u^{p-1}}{f(v)}|\nabla u||\nabla v|^{p-1}+\frac{u^{p} f^{'}(v)}{(f(v))^2}|\nabla v|^p)\\
+w(x)\frac{pu^{p-1}}{f(v)}(|\nabla u||\nabla v|^{p-1}-|\nabla v|^{p-2}\nabla u.\nabla v)
\end{multline*}
Applying Young's Inequality with $|\nabla u|$ and $u|\nabla v| (f(v))^\frac{1}{1-p}$,
we get $L(u,v)\geq 0$.\\
Following the exact proof of Allegretto-Huang \cite{AlHu}, we have that the equality holds iff $u=cv+d$ for some constants $c$ and $d$. 
\end{proof}
\begin{Rem}
For $p=2$ and $w=1$, we get back Theorem 1.1 of Tyagi \cite{JT} and rectifies Theorem 2.1 of Bal \cite{KB} for
$p>1$.
\end{Rem}
\begin{Lem}\label{uselemma}
For $\phi\in [C_{c}^\infty(\Omega)]^{+}$ and $v\in{C}(\Omega)\cap W^{1,p}(\Omega,w)$ such that $v>0$ in $\Omega$, we have $\frac{\phi^p}{h(v)}\in X^{+}$, provided $w\in\mathcal{A}_p$.
\end{Lem}
\begin{proof}
Since $h\in C^1(0,\infty)$ and $v$ is continuous on any compact subset $K$ of $\Omega$ we have the functions
$$
F_{1}:K\rightarrow\mathbb{R}\;\mbox{defined \;by}\; F_{1}(x)=h'(v)(x)
$$ and
$$
F_{2}:K\rightarrow\mathbb{R}\;\mbox{defined \;by}\; F_{2}(x)=h(v)(x)
$$ are continuous.
Therefore on any compact subset $K:=\mbox{supp}\;\phi$ we have 
\begin{equation}\label{F_1}
|F_{1}(x)|\leq c
\end{equation} where $c$ is a constant independent of $x$ and 
\begin{equation}\label{F_2}
|F_{2}(x)|\geq\delta>0
\end{equation} where $\delta$ is a constant independent of $x$. Now using (\ref{F_2}), we obtain 
$$
\int_{\Omega}\Big|\frac{\phi^p}{h(v)}\Big|^p\;dx\leq \frac{||\phi||_{\infty}^{p^2}}{\delta^p}|\Omega|<+\infty.
$$
From (\ref{F_1}) and (\ref{F_2}), we get
\begin{align*}
\int_{\Omega} w(x)|\nabla(\frac{\phi^p}{h(v)})|^p\;dx
&=\int_{K} w(x)
\Big|\frac{p{\phi}^{p-1} \nabla\phi}{h(v)}-\frac{h'(v){\phi}^p \nabla v}{(h(v))^2}\Big|^p\;dx\\
&\leq 2^p \int_{K} w(x)(p^p\frac{|\phi|^{p(p-1)}|\nabla\phi|^p}{h(v)^p}\\
&+(\frac{h'(v)}{(h(v))^2})^p|\phi|^{p^2}|\nabla v|^p)\;dx\\
&\leq 2^p \int_{K} w(x)(p^p\frac{||\phi||_{\infty}^{p(p-1)}}{\delta^p}|\nabla\phi|^p
+\frac{c^{p}||\phi||_{\infty}^{p^2}}{\delta^{2p}}|\nabla v|^p)\;dx\\
&\leq C(||\phi||_{1,p,w}^p+||v||_{1,p,w}^p)\;dx<+\infty.
\end{align*}
Hence the lemma follows.
\end{proof}
\begin{Lem}\label{newlemma}
Let $w\in \mathcal{A}_s$ with $p_s>N$. Suppose $u\in X^+$ be such that supp $u\subset K$ for some compact subset $K$ of $\Omega$ and $v\in W^{1,p}(\Omega,w)$ positive in $\Omega$. Then $\frac{u^p}{h(v)}\in X^+$.
\end{Lem}
\begin{proof}
Since $u,v\in W^{1,p}(\Omega,w)$ with $w\in A_s$ with $p_s>N$ by Lemma $\ref{embedding theorem}$, we may assume both $u$ and $v$ are continuous upto the boundary. 
Since $h\in C^1(0,\infty)$ both $h(v)$ and $h'(v)$ are continuous over any compact subset $K$ of $\Omega$. Suppose supp $u\subset K$. 
Let us assume that 
\begin{equation}\label{bound}
||u||_\infty\leq M
\end{equation}
Let $\delta$,$T$ be constants such that 
\begin{equation}\label{useful}
h(v)\geq\delta>0\;\;
\text{and}\;\;
h'(v)\leq T\;\;\text{on}\;\;K
\end{equation}
Now using (\ref{bound}) and (\ref{useful}), we obtain
\begin{align*}
\int_{\Omega}|\frac{u^p}{h(v)}|^p dx
\leq\int_{K}\frac{||u||_{\infty}^{p^2}}{|h(v)|^p}\;dx\leq \frac{M^{p^2}}{\delta^{p}}|K|<+\infty.
\end{align*}
and
\begin{align*}
\int_{\Omega}w(x)|\nabla(\frac{u^p}{h(v)})|^p dx
&=\int_{K}w(x)\Big|p\frac{u^{p-1}\nabla u}{h(v)}-\frac{h'(v)}{(h(v))^2}u^p \nabla v\Big|^p\;dx\\
&\leq 2^p\int_{K}w(x)\{p^p\frac{|u|^{p(p-1)}}{(h(v))^p}|\nabla u|^p+\frac{(h'(v))^p}{(h(v))^{2p}}|u|^{p^2}|\nabla v|^p\}\;dx\\
&\leq 2^p\int_{K}w(x)\{p^p\frac{M^{p(p-1)}}{\delta^p}|\nabla u|^p+\frac{T^p}{\delta^{2p}}M^{p^2}|\nabla v|^p\}\;dx\\
&\leq C \{||u||^p_{1,p,w}+||v||^p_{1,p,w}\}\\
&< +\infty.
\end{align*}
Hence the lemma follows.
\end{proof}

For Lemma \ref{uselemma0} and Lemma \ref{uselemma1}, we assume $h\in\mathbb{Q}$ and is monotone increasing satisfying the following property:
\begin{equation}\label{subcon}
q_{\epsilon}:=\big|\frac{h'(x+\epsilon)}{(h(x+\epsilon))^2}\big|_{\infty}\leq M(\epsilon)\;\mbox{for any}\;\epsilon>0
\end{equation}
where $M(\epsilon)$ is independent of $x$.
Clearly $e^{(p-1)x}$ and $(x+c)^{p-1}$ is in the above class for any $c>0$ and $p>1$.

\begin{Lem}\label{uselemma0}
Let $w\in \mathcal{A}_p$, $\phi\in [C_{c}^{\infty}(\Omega)]^+$ and $v\in [W^{1,p}(\Omega,w)]^+$. Then we have $\frac{\phi^p}{h(v+\epsilon)}\in X^{+}$ for every $\epsilon>0$.
\end{Lem}

\begin{proof}
Since $h$ is increasing, we have
$$
\int_{\Omega}\Big|\frac{\phi^p}{h(v+\epsilon)}\Big|^p<(\frac{||\phi||_{\infty}^{p}}{h(\epsilon)})^p|\Omega|<+\infty.
$$
Moreover,
\begin{multline*}
\int_{\Omega} w(x)|\nabla(\frac{\phi^p}{h(v+\epsilon)})|^p\;dx
=\int_{\Omega} w(x)\Big|\frac{p{\phi}^{p-1} \nabla\phi}{h(v+\epsilon)}-q_{\epsilon}(v){\phi}^p {\nabla v}\Big|^p\;dx\\
\leq 2^p\int_{\Omega} w(x)(\frac{p^{p}||\phi||_{\infty}^{p(p-1)}}{(h(\epsilon))^p}|\nabla\phi|^p+M^p({\epsilon})||\phi||_{\infty}^{p^2}|\nabla v|^p)\;dx\\
\leq C(||\phi||^{p}_{1,p,w}+||v||^{p}_{1,p,w})<+\infty.
\end{multline*}
since $\phi\in C_{c}^{\infty}(\Omega)$, the lemma follows.
\end{proof}

\begin{Lem}\label{uselemma1}
Let $w\in \mathcal{A}_s$ with $p_s>N$. Then we have $\frac{u^p}{h(v+\epsilon)}\in X^{+}$ provided $u\in X^+$ and $v\in [W^{1,p}(\Omega,w)]^+$.
\end{Lem}
\begin{proof} 
$w\in \mathcal{A}_s$, $p_s>N$ by Lemma $\ref{embedding theorem}$, both $u$ and $v$ are continuous upto the boundary.
Therefore,
$$
\int_{\Omega}|\frac{u^p}{h(v+\epsilon)}|^p dx\leq \frac{||u||_{\infty}^{p^2}}{h(\epsilon)^p}|\Omega|<+\infty.
$$
and
\begin{multline*}
\int_{\Omega} w(x)\Big|\nabla(\frac{u^p}{h(v+\epsilon)})\Big|^p\;dx
=\int_{\Omega} w(x)\Big|\frac{pu^{p-1}\nabla u}{h(v+\epsilon)}-q_{\epsilon}(v)u^p\nabla v\Big|^p\;dx\\
\leq 2^p \int_{\Omega}w(x)\{\frac{p^p||u||_{\infty}^{p(p-1)}|\nabla u|^{p}}{(h(\epsilon))^p}
+M^p({\epsilon})||u||_{\infty}^{p^2}|\nabla v|^p\}\;dx\\
\leq C(||u||_{1,p,w}^{p}+||v||^{p}_{1,p,w})
<+\infty.
\end{multline*}
Hence the lemma follows.
\end{proof}

\section{Main Results}
We start this section by stating our main results.\\
For $u>0$ in $\Omega$, consider the equation
\begin{equation}\label{np1}
-\Delta_{p,w}u-g_1(x)f(u)=g_2(x);\;\;w\in \mathcal{A}_s;\;\;p_s>N
\end{equation}
where $0\leq{g_1,g_2}\in{L^{1}(\Omega)}$ and $f\in\mathbb{M}$ satisfying (\ref{subcon}).
We say $u\in{W^{1,p}(\Omega,w)}$ is a positive weak super-solution of equation (\ref{np1}), if $u>0$ in $\Omega$ and 
\begin{equation}\label{np2}
\begin{gathered}
\int_{\Omega}w(x)|\nabla u|^{p-2}\nabla u.\nabla \varphi dx-\int_{\Omega}g_1(x)f(u)\varphi dx \geq {\int_{\Omega}g_2(x)\phi dx}
\end{gathered}
\end{equation}
for all $\varphi\in X^{+}$.
\begin{The}\label{np}(Non existence of positive super-solutions)
Given $w\in \mathcal{A}_s$ with $p_{s}> N$ if there exist $u\in X^+$ such that $$J(u)=\int_{\Omega}w(x)|\nabla u|^p dx-\int_{\Omega}g_1(x) u^p dx<0,$$ then (\ref{np1}) has no positive weak super-solution in $W^{1,p}(\Omega,w)$.
\end{The}
\begin{Rem}\label{NP}
Assume that $f\in\mathbb{M}$ satisfying (\ref{subcon}) and $w\in \mathcal{A}_s$ with $p_s>N$ where $g_1,g_2\in L^1(\Omega)$ be non-negative. If there exist a function $u\in X^+$ with supp $u\subset K$ for some compact subset $K$ of $\Omega$ such that $J(u)<0$,
then (\ref{np1}) has no positive weak super-solution in $W^{1,p}(\Omega,w)$.   
\end{Rem}

\begin{The}\label{har}(Hardy-type Inequality)
Let $w\in \mathcal{A}_p $ and assume that there exists $v\in C(\overline{\Omega})\cap W^{1,p}(\Omega,w)$ or $v\in C^1(\overline{\Omega})$ satisfying
\begin{equation}\label{evp}
-\Delta_{p,w} v\geq  g(x)f(v);\quad v>0\quad\text{in}\quad\Omega;\;f\in\mathbb{M}
\end{equation}
for some non-negative function $g\in L^1(\Omega)$. Then for any $u\in [C^\infty_c(\Omega)]^+$ it holds that
\begin{equation}\label{Hardy}
\int_\Omega w(x)|\nabla u|^p dx\geq\int_\Omega g(x)u^p dx.
\end{equation}
\end{The}
\begin{Rem}\label{Hardy remark}
Assume $w\in \mathcal{A}_s$ with $p_s>N$ and $g\in L^1(\Omega)$ be non-negative such that there exists $v\in W^{1,p}(\Omega,w)>0$ in $\Omega$ satisfying (\ref{evp}), then for any $u\in X^+$ with support $u\subset K$ for some compact subset $K$ of $\Omega$, the inequality (\ref{Hardy}) holds.
\end{Rem}

\begin{The}\label{sc}(Comparison Theorem)
Let $w_1\in \mathcal{A}_p$ be a weight function and $w_{2}\in A_s$ with $p_s>N$ such that $w_2(x)\leq w_1(x)$ for a.e. $x\in\Omega$. Assume $u\in [W^{1,p}(\Omega,w_1)]^{+}$ be such that 
\begin{equation}\label{sc1}
-\;\Delta_{p,w_{1}}u=f_1(x)|u|^{p-2}u;\quad u\geq 0(\not\equiv 0)\;\;\text{in}\;\;\Omega;\;\; u=0\;\;\text{on}\;\;\partial\Omega
\end{equation}
Then any non-trivial solution $v\in W^{1,p}(\Omega,w_2)$ of
\begin{equation}\label{sc2} 
-\Delta_{p,w_{2}}v=f_2(x)|v|^{p-2}v\;\;\text{in}\;\;\Omega;\;\;v=0\;\text{on}\;\;\partial\Omega
\end{equation}
where $0\leq f_1(x)< f_2(x)$ a.e. $x\in{\Omega}$ and $f_1,f_2\in L^{1}(\Omega)$ must change sign.
\end{The}

We conclude this section with a brief discussion of the eigenvalue problem. Consider the eigenvalue problem 
\begin{equation}\label{ev}
\begin{gathered}
-\Delta_{p,w}u={\lambda}\beta(x)|u|^{p-2}u\quad\text{in}\quad\Omega;\\
u=0\quad\text{on}\quad\partial{\Omega}.
\end{gathered}
\end{equation}
where $w\in \mathcal{A}_t$ and $\beta(\geq 0)\in L^{\infty}(\Omega)$ or $\beta(\geq 0)\in L^{\frac{q}{q-p}}(\Omega)$ for some $q\in(p,p_s^{*})$, where $p_s^{*}=\frac{Np_s}{N-p_s}$ with $1\leq p_s<N$. Moreover, let
$$|\{x\in\Omega:\; \beta(x)>0\}|>0$$
Here we state a result, proof of which can be found in Chapter 1 of Dr\'abek et al \cite{Ak}, which in turn will give a new norm on the space $X$.
\begin{Lem}\label{FI}(The weighted Friedrich inequality)
For $w\in \mathcal{A}_t$ the inequality
\begin{equation}\label{FIeqn}
\int_{\Omega}|u|^p dx\leq c\int_{\Omega}w(x)|\nabla u|^p dx
\end{equation}
holds for every $u\in X$ with the constant $c>0$ independent of $u$.
\end{Lem}
\begin{Def}
A real number $\lambda$ such that $(\ref{ev})$ admits a non trivial solution $u$ is called an eigenvalue of the operator $-\Delta_{p,w}$ and $u$ is called the corresponding eigenfunction. 

We denote the principal eigenvalue of (\ref{ev}) as $\lambda_{1}$ and is defined as:
$$\lambda_{1}=\inf\{\int_{\Omega}w(x)|\nabla{v}|^p\;dx:\int_{\Omega}\beta(x)|v|^p\;dx=1\}$$
\end{Def} 
Observe that due to Lemma \ref{FI}, the space $X$ can be defined via an equivalent norm 
$$
||u||_X=(\int_{\Omega}w(x)|\nabla u|^p dx)^\frac{1}{p}
$$ for every $u\in X$.
We now turn to results related to the eigenvalue problem, the existence of which is already available in \cite{Ak}. The simplicity of $\lambda_1$ is proved in a different way by Dr\'abek et al \cite{Ak} but here we provide a simpler proof using ideas from Belloni-Kawohl \cite{BeKa}.

\begin{The}\label{sbh}(Simplicity of the first eigenvalue)
The first eigenvalue of the operator $-\Delta_{p,w}$ is simple for any $w\in \mathcal{A}_t$.
\end{The}

\begin{The}\label{MP}(Monotonicity property of the first eigenvalue)
Let $w\in \mathcal{A}_t$ and suppose $\Omega_{1}\subset{\Omega_{2}}$ such that $\Omega_{1}\neq{\Omega_{2}}$. Then $\lambda_{1}(\Omega_{1})>\lambda_{1}(\Omega_{2})$.
\end{The}

\begin{The}\label{liop}(Liouville Theorem)
Let $w\in \mathcal{A}_t$ with $p_s>N$. Then the inequality given by $$-\Delta_{p,w} u\geq C\beta(x)|u|^{p-2}u$$ where $\beta(x)\in [L^{\infty}(\mathbb{R}^N)]^{+}$ does not admit a positive solution in $W^{1,p}_{loc}(\mathbb{R}^N,w)$.
\end{The}

\section{Proofs of Main Theorems}

\begin{proof}[Proof of Theorem \ref{np}]
Suppose $v$ is a positive weak super-solution of (\ref{np1}).\\
Choosing $\varphi=\frac{u^p}{f(v+\epsilon)}\in{X}$ as a test function in (\ref{np2}) admissible by Lemma \ref{uselemma1}, and using Lemma (\ref{GPI}), we get
\begin{align*}
\int_{\Omega}g_1(x)\frac{f(v)}{f(v+\epsilon)}u^pdx
&\leq\int_{\Omega}\Big[w(x)|\nabla v|^{p-2}\nabla v.\nabla (\frac{u^p}{f(v+\epsilon)})-g_2(x)\frac{u^p}{f(v+\epsilon)}\Big]dx\\
&=\int_{\Omega}\Big[w(x)|\nabla u|^p-w(x)|\nabla u|^p\\
&+w(x)|\nabla v|^{p-2}\nabla v.\nabla\Big(\frac{u^p}{f(v+\epsilon)}\Big)
-g_2(x)\frac{u^p}{f(v+\epsilon)}\Big]dx\\
&=\int_{\Omega}\Big[w(x)|\nabla u|^p-\int_{\Omega}R(u,v+\epsilon)-\int_{\Omega}g_2(x)\frac{u^p}{f(v+\epsilon)}\Big]dx
\end{align*}
Then Fatou's lemma yields $$J(u)\geq\int_{\Omega}R(u,v)\geq 0$$
Hence the theorem.
Remark \ref{NP} follows similarly by choosing $\frac{u^p}{f(v)}$ as a test function admissible by Lemma (\ref{newlemma}).
\end{proof}

\begin{proof}[Proof of Theorem \ref{har}]
Using $\frac{u^p}{f(v)}\in X$ as a test function in (\ref{evp}) which is admissible by Lemma \ref{uselemma}, we have
\begin{equation}\label{HD}
\int_{\Omega}g(x)u^p dx\leq\int_{\Omega}w(x)|\nabla v|^{p-2}\nabla v.\nabla \Big(\frac{u^p}{f(v)}\Big)dx.
\end{equation}
By Lemma \ref{GPI} we have 
\begin{equation}\label{hd}
\int_{\Omega}R(u,v)dx=\int_{\Omega} w(x)(|\nabla u|^p-|\nabla v|^{p-2}\nabla v.\nabla\Big(\frac{u^p}{f(v)}\Big)dx\geq 0
\end{equation}
Using the inequality (\ref{HD}) in (\ref{hd}) we get the inequality (\ref{Hardy}). Hence the theorem follows.\\
Remark \ref{Hardy remark} follows by choosing $\frac{u^p}{f(v)}$ as a test function admissible by Lemma (\ref{newlemma}).
\end{proof}

\begin{proof}[Proof of Theorem \ref{sc}]
Since $w_2\leq w_1$ and $u\in W^{1,p}(\Omega,w_1)$ we have
$$
\int_{\Omega}|u|^p dx+\int_{\Omega}w_{2}|\nabla u|^p dx\leq\int_{\Omega}|u|^p dx+\int_{\Omega}w_{1}|\nabla u|^p dx<+\infty
$$
which implies $u\in W^{1,p}(\Omega,w_2)$.
Without loss of generality let $v>0$ a.e. in $\Omega$. The case $v<0$ can be dealt similarly by considering the function $-v$. Since $w_2\in \mathcal{A}_s$
with $p_s>N$ using $\frac{u^p}{(v+\epsilon)^{p-1}}\in W_{0}^{1,p}(\Omega,w_2)$ as a test function in (\ref{sc2}) admissible by Lemma (\ref{uselemma1}) we get
\begin{equation}\label{ST}
\int_{\Omega} w_{2}|\nabla v|^{p-2}\nabla v.\nabla\Big(\frac{u^p}{(v+\epsilon)^{p-1}}\Big)dx=\int_{\Omega} f_{2}(x)\Big(\frac{v}{v+\epsilon}\Big)^{p-1}u^p dx
\end{equation}
Taking $u$ as a test function in (\ref{sc1}) we get
\begin{equation}\label{st}
\int_{\Omega} w_{1}|\nabla u|^p dx=\int_{\Omega}f_{1}(x)u^p dx
\end{equation}
By Lemma \ref{GPI} and using (\ref{ST}), (\ref{st}), we get
\begin{multline*}
0\leq\int_{\Omega}L(u,v+\epsilon)dx
=\int_{\Omega}R(u,v+\epsilon)dx\\
=\int_{\Omega}w_1(x)|\nabla u|^p dx-\int_{\Omega}w_1(x)|\nabla v|^{p-2}\nabla v.\nabla \Big(\frac{u^p}{(v+\epsilon)^{p-1}}\Big)dx\\
\leq\int_{\Omega}w_1(x)|\nabla u|^p dx-\int_{\Omega}w_2(x)|\nabla v|^{p-2}\nabla v.\nabla\Big (\frac{u^p}{(v+\epsilon)^{p-1}}\Big)dx\\
=\int_{\Omega}f_1(x)u^p dx-\int_{\Omega}f_2(x)(\frac{v}{v+\epsilon})^{p-1}u^p dx
\end{multline*}
Using Fatou's lemma we get 
$$\int_{\Omega}(f_1(x)-f_2(x))u^p dx\geq 0$$
which gives a contradiction since $f_1<f_2$ and $u\geq 0$. Hence $v$ must change sign.
\end{proof}

\begin{proof}[Proof of Theorem \ref{sbh}]
Note that the first eigenvalue $\lambda_1$ of the operator $-\Delta_{p,w}$ is given by the minimizer of the functional
$$
J_{p,w}(v)=\int_{\Omega}w(x)|\nabla v|^p dx\;\;\text{on}\;\;K=\{v\in X:\int_{\Omega}\beta(x)|v|^p dx=1\}.
$$
Let us suppose $u_1$ and $u_2$ be two non-negative eigenfunction corresponding to the first eigenvalue $\lambda_1$. Therefore we have
$$
\int_{\Omega}\beta(x)|u_1(x)|^p dx=1
$$ and
$$
\int_{\Omega}\beta(x)|u_2(x)|^p dx=1.
$$
Choosing $u_3(x)=U^{\frac{1}{p}}(x)$ for $U(x)=\frac{u_1^p(x)+u_2^p(x)}{2}$ we get $\int_{\Omega}\beta(x)|u_3(x)|^p dx=1.$
Now for $t(x)=\frac{u_1^p}{u_1^p+u_2^p}(x)\in(0,1)$, we have
\begin{align*}
w(x)|\nabla u_3|^p 
&=w(x)U^{1-p}|\frac{1}{2}(u_1^{p-1}\nabla u_1+u_2^{p-1}\nabla u_2)|^p\\
&=w(x)U|\frac{1}{2}(\frac{u_1^p}{U}\frac{\nabla u_1}{u_1}+\frac{u_2^p}{U}\frac{\nabla u_2}{u_2})|^p\\
&=w(x)U|t(x)\frac{\nabla u_1}{u_1}+(1-t(x))\frac{\nabla u_2}{u_2}|^p\\
&\leq w(x)U\{t(x)|\frac{\nabla u_1}{u_2}|^p+(1-t(x))|\frac{\nabla u_1}{u_1}|^p\}\\
&=\frac{w(x)}{2}(|\nabla u_1|^p+|\nabla u_2|^p)\\
&=\frac{1}{2}(w(x)|\nabla u_1|^p+w(x)|\nabla u_2|^p)
\end{align*}
Therefore we obtain
\begin{equation}\label{ineq}
\int_{\Omega} w(x)|\nabla u_3(x)|^p dx\leq\frac{1}{2}(\int_{\Omega} w(x)|\nabla u_1(x)|^p dx+\int_{\Omega} w(x)|\nabla u_2(x)|^p dx)
\end{equation} 
Note that since both $u_1$ and $u_2$ are minimizers the equality in (\ref{ineq}) must hold i.e,
$$
\frac{\nabla u_1}{u_1}=\frac{\nabla u_2}{u_2}\;\;\text{a.e.in}\;\;\Omega
$$ which implies $$u_1=cu_2\;\;\text{a.e.in}\;\;\Omega$$ Hence $\lambda_1$ is simple.
\end{proof}

\begin{proof}[Proof of Theorem \ref{MP}]
Note that $C_{c}^{\infty}(\Omega_{1})$ is dense in $W_{0}^{1,p}(w,\Omega_1)$ provided $w\in\mathcal{A}_p$. The assumption $w\in\mathcal{A}_t$ is required to guarantee the existence of first eigenfunction. Using the above information and putting $f(x)=x^{p-1}$ in Lemma \ref{GPI} as in Allegretto-Huang \cite{AlHu} we have our result.  
\end{proof}

\begin{proof}[Proof of Theorem \ref{liop}] 
Note that from Drabek et al \cite{Ak}, there exists the least an eigenvalue $\lambda_1>0$ and at least one corresponding eigenfunction $u_1\geq 0$ a.e. in $\Omega\;(u_1\not\equiv 0)$ of the eigenvalue problem (\ref{ev}). Also from Theorem \ref{MP} we have 
$$\lambda_1(\Omega_2)<\lambda_1(\Omega_1)\;\;\mbox{provided}\quad\Omega_1\subset\Omega_2\;\;\text{with}\;\;\Omega_1\neq\Omega_2$$
Given $C>0$ one can choose a sufficiently large ball $B$ of radius $r>0$ and center at $y$ such that $\lambda_1(B)<C$, where $\phi$ is the non-negative eigenfunction corresponding to the first eigenvalue $\lambda_1(B)$ . 

Therefore we have
$$
-\Delta_{p,w}\phi=\lambda_1(B)\beta(x)\phi^{p-1}
$$
Let $u$ be a positive solution of 
\begin{equation}\label{B}
-\Delta_{p,w} u\geq C\beta(x)u^{p-1}\;\;\text{in}\;\;B
\end{equation}
and by Lemma \ref{uselemma1} choosing $\frac{\phi^p}{(u+\epsilon)^{p-1}}\in W_{0}^{1,p}(B,w)$ as a test function in (\ref{B}) we obtain after using Lemma \ref{GPI}
\begin{multline*}
C\int_{B}\beta(x)(\frac{u}{u+\epsilon})^{p-1}\phi^p dx-\int_{B}w(x)|\nabla\phi|^p dx\\
\leq\int_{B}w(x)|\nabla u|^{p-2}\nabla{u}.\nabla(\frac{\phi^p}{(u+\epsilon)^{p-1}})dx-\int_{B}w(x)|\nabla u|^p dx
\leq 0
\end{multline*}
Letting $\epsilon\to 0^+$ using Fatou's Lemma we get
$$
C\leq\frac{\int_{B}w(x)|\nabla\phi|^p}{\int_{B}\beta(x)\phi^p}dx=\lambda_{1}(B)<C
$$ which is a contradiction. This completes the proof. 
\end{proof}

\section*{Acknowledgement:}
The author was supported by NBHM Fellowship No: 2-39(2)-2014-NBHM-RD-II-8020-June 26, 2014. The author would like to thank Dr. Kaushik Bal for the fruitful discussions and suggestions on the topic.\\
On behalf of all authors, the corresponding author states that there is no conflict of interest.

Prashanta Garain\\
Department of Mathematics and Statistics\\
Indian Institute of Technology, Kanpur\\
Uttar Pradesh-208016, India\\
e-mail: pgarain@iitk.ac.in

\begin{thebibliography}{0}

\bibitem{AlHu} Allegretto, W; Yin, X H.;
\emph{A {P}icone's identity for the {$p$}-{L}aplacian and applications},
Nonlinear Anal., 1998, Vol. 32(7), pp. 819--830.

\bibitem{KB} Bal, K;
\emph{Generalized Picone's Identity and its Applications},
Electron. J. Differential Equations., 2013, Vol.  243, pp. 6.
 


\bibitem{BeKa} Belloni; M, Kawohl, B
\emph{ A direct uniqueness proof for equations involving the $p$-Laplace operator},
Manuscripta Math,109-2 (2002), 229-231.
      

\bibitem{ChWe} Chanillo, S; Wheeden, R
\emph{Weighted {P}oincar\'e and {S}obolev inequalities and estimates
              for weighted {P}eano maximal functions},
Amer. J. Math, Vol =107, 1985, (5) pp-1191--1226

\bibitem{DeVi} De Cicco, V; Vivaldi, Maria A
A Liouville type theorem for weighted elliptic equations. 
Adv. Math. Sci. Appl. 9 (1999), no. 1, 183–207

\bibitem{Ak} Dr\'abek, P; Kufner, A; Nicolosi, F;
\emph{Quasilinear elliptic equations with degenerations and singularities}
De Gruyter Series in Nonlinear Analysis and Applications Walter de Gruyter and Co., Berlin, 1997. xii+219 pp                                             

\bibitem{EFabes} Fabes, E. B.; Kenig, C. E.; Serapioni, R. P.;
\emph{The local regularity of solutions of degenerated elliptic equations},
Commun. Partial Diff. Eq., 1982, Vol. 7, No. 1, pp. 77-116.

\bibitem{JTO} Heinonen, J.; Kilpelainen, T.; Martio, O.;
\emph{Nonlinear potential theory of degenerate elliptic equations},
Clarendon Press, 1993.
 

\bibitem{KLP} Kawohl, B.; Lucia, M.; Prashanth, S.;
\emph{Simplicity of the principal eigenvalue for indefinite quasilinear problems},Adv. Differential Equations., 2007, Vol. 12, pp. 407-434.


\bibitem{JT}  Tyagi, J;
\emph{A nonlinear Picone's identity and its applications}
 Appl. Math. Lett. 26 (2013), no. 6, 624-626.



\end{thebibliography}
\end{document}